\documentclass[11pt,reqno]{amsart} % please use amsart at 11pt

\usepackage{amssymb,latexsym}

\usepackage[height=190mm,width=130mm]{geometry} % this is the journal
\theoremstyle{plain}
\newtheorem{theorem}{Theorem}
\newtheorem{lemma}{Lemma}
\newtheorem{corollary}{Corollary}

\theoremstyle{definition}
\newtheorem{definition}{Definition}

\theoremstyle{remark}
\newtheorem{remark}{Remark}
\newtheorem{example}{Example}

\numberwithin{equation}{section} % to get equations numbered
\begin{document}
\title[Unified treatment of fractional integral inequalities]{Unified treatment of fractional integral inequalities via linear functionals}

\author{M. Bombardelli}
\address{Department of Mathematics,  University of Zagreb\\Zagreb, Croatia}

\email{bombarde@math.hr}

\author{L. Nikolova}
\address{Department of Mathematics and Informatics,
          Sofia University\\Sofia, Bulgaria}

\email{ludmilan@fmi.uni-sofia.bg}

\author{S. Varo\v sanec}
\address{Department of Mathematics,  University of Zagreb\\Zagreb, Croatia}

\email{varosans@math.hr}            
                    
 \thanks{} 
  
\begin{abstract}  
In the paper we prove several inequalities involving two isotonic linear functionals. We consider inequalities for functions with variable bounds, for Lipschitz and H\" older type functions etc. 
These results give us an elegant method for obtaining a number of inequalities for various kinds of fractional integral operators 
 such as for the Riemann-Liouville fractional integral operator, the Hadamard fractional integral operator, fractional hyperqeometric integral and corresponding q-integrals.
\end{abstract}
 
\subjclass{26D10; 26A33}

\keywords{the Chebyshev inequality, the Chebyshev difference, fractional integral operator, isotonic linear functional, Lipschitz function}
 
 \maketitle

  \section{Introduction}
 %In the last few decades we are witnesses of fast development  of fractional calculus theory.  We pay our attention to inequalities connected with different fractional integral operators.   
Recently several papers involving inequalities for fractional integral operators  have been published, see \cite{BPA,DAH2,CP1,CP,SNT,TNS,WHPG} and references therein.
Certain similarity of those inequalities shows that those results have a common origin.
In this paper we give a unified treatment of several known inequalities for fractional integral operators via theory of isotonic linear functionals.
In fact, we prove general inequalities involving isotonic linear functionals from which some interesting results are followed.

  The paper is organized in the following way. The rest of this section contains definitions and some examples of isotonic linear functionals connected with fractional integration and integration on time scales.  The Chebyshev  inequality for one and two isotonic functionals are given. Inequalities for Lipschitz functions are given in the  second section. The third  section is devoted to  new  inequalities involving two isotonic functionals and functions with variable upper and lower bounds.  The fourth section is devoted to results involving more than  two functions. 
Applications or references where we can find applications in the theory of fractional operators and calculus on time scales  are also given.

\subsection*{Isotonic linear functionals}
\begin{definition} ({\it Isotonic linear functional}) Let $E$ be a non-empty set and $L$ be a
class of real-valued functions on $E$ having the properties:

L1. If $f,g\in L$, then $(af+bg)\in L$ for all $a,b\in {\bf R}$;

L2. The function ${\bf 1}$ belongs to $L$. (${\bf 1}(t)=1$ for $t\in E$).

\noindent A functional $A: L \rightarrow {\bf R}$ is called an isotonic linear functional if

A1. $A(af+bg)=aA(f)+bA(g)$ for $f,g \in L$, $a,b\in {\bf R}$;

A2. $f \in L$, $f(t)\geq 0$ on $E$ implies $A(f)\geq 0$.
\end{definition}

%A lot of results involving isotonic linear functionals connected with the Jensen, the H\" older, the Minkowski and related inequalities are given in monograph \cite{proschan}. Results about Chebyshev-type inequalities for isotonic linear functionals can be found in separated papers such as in  \cite{NV,PT1} and references therein.
  
There  exist a lot of interesting examples of linear functionals which play some role in different parts of mathematics. In the following example we describe the most mentioned functionals - discrete and integral, and we give several functionals which appear in the theory of fractional calculus and calculus on time scales.

\begin{example}
(i) {\bf Discrete functional.}  If $E=\{1,2,\ldots ,n\}$ and $f:E\rightarrow {\bf R}$, then $A(f)=\sum_{i=1}^n f(i)$ is an isotonic linear functional.

%\vspace{1ex}

(ii) {\bf Integral functional.} If $E=[a,b]\subset {\bf R}$ and $L=L(a,b)$, then
$A(f)=\int_a^b f(t) dt$ is an isotonic linear functional.  If $\displaystyle A_1(f)=\frac{1}{b-a}A(f)$, then $A_1$ is a normalized isotonic linear functional.

%\vspace{1ex}

(iii) {\bf Fractional hypergeometric operator. } If  $t>0, \ \alpha>\max\{0,-\beta-\mu\},\mu>-1, \beta-1<\eta<0$, then
\[A(f)=I^{\alpha,\beta,\eta,\mu}_t \{f(t)\}\]
 is an isotonic linear functional, (\cite{BPA}), where
$I^{\alpha,\beta,\eta,\mu}_t \{f(t)\}$ is a fractional hypergeometric operator defined as
\begin{multline*}  
I^{\alpha,\beta,\eta,\mu}_t \{f(t)\} \\
 =\frac{t^{-\alpha-\beta-2\mu}}{\Gamma(\alpha)}\int_0^t\sigma^\mu (t-\sigma)^{\alpha-1}
\ _2F_1(\alpha+\beta+\mu, -\eta,\alpha;1-\frac{\sigma}{t})f(\sigma)\,d\sigma 
\end{multline*}
where the function $\displaystyle _2F_1(a,b,c,t)=\sum_{n=0}^\infty \frac{(a)_n (b)_n}{(c)_n}\frac{t^n}{n!}$
  is the Gaussian hypergeometric function  and $(a)_n$ is the Pochhammer symbol:
$(a)_n=a(a+1)\ldots(a+n-1), \ \ (a)_0=1$.

%\vspace{2ex}

\begin{itemize}
\item %{\bf The Saigo fractional integral operator}
Putting  $\mu =0$, then the fractional hypergeometric operator reduces to the Saigo fractional integral operator $I^{\alpha,\beta,\eta} \{f(t)\}$. 

\item The Erd\' elyi-Kober fractional integral operator $I^{\alpha,\eta} \{f(t)\}$ is a particular case of $I^{\alpha,\beta,\eta,\mu}_t \{f(t)\}$ when $\beta=\mu=0$.
  
\item One of the earliest defined and the most  investigated fractional integral operator is the so-called
Riemann-Liouville   operator defined as
\begin{eqnarray} \label{riemann}
J^{\alpha} f(t)=I^{\alpha,-\alpha,0,0}_t \{f(t)\}
=\frac{1}{\Gamma(\alpha)}\int_0^t (t-\sigma)^{\alpha-1}
f(\sigma)\,d\sigma, \ \ \alpha>0, &&
\end{eqnarray}
and it is a particular case of a fractional hypergeometric operator for $\beta =-\alpha$, $\eta=\mu =0$.
\end{itemize}

(iv) {\bf q-analogues}
The above-mentioned operators have the so-called $q$-analogues. We describe a $q$-analogue of Saigo's fractional integral, \cite{BA}. Let $\Re (\alpha) >0$, $\beta, \eta \in {\bf C}$, $0<q<1$. A $q$-analogue of Saigo's fractional integral $I_q^{\alpha,\beta,\eta}$ is given for $|\tau/t|<1$ by equation
$$I_q^{\alpha,\beta,\eta}\{f(t)\}=\frac{t^{-\beta-1}}{\Gamma_q(\alpha)} \int_0^t \left( q\frac{\tau}{t};q\right)_{\alpha-1}\times$$
$$\times \sum_{m=0}^{\infty}\frac{(q^{\alpha+\beta};q)_m 
(q^{-\eta};q)_m}{(q^{\alpha };q)_m (q ;q)_m}
\cdot q^{(\eta-\beta)m}(-1)^mq^{-m(m-1)/2}\left(\frac{\tau}{t}-1\right)^m_q f(\tau)d_q\tau,
$$
where
$$(a;q)_{\alpha}=\frac{\prod_{k=0}^\infty (1-aq^k)}{\prod_{k=0}^\infty (1-aq^{\alpha+k})} \ \
\textrm{and} \ \  (t-a)_q^n=t^n(\frac at; q)_n.$$
If $\alpha >0$, $\beta, \eta \in {\bf R}$ with $\alpha+\beta >0$ and $\eta<0$, then $I_q^{\alpha,\beta,\eta}$ is isotonic, \cite{choi}.

%\vspace{1ex}

(v) {\bf The Hadamard fractional integral}

The Hadamard fractional integral of order $\alpha>0$ of function $f$ is defined as
 $$ _HJ^\alpha f (x)= \frac{1}{\Gamma(\alpha)} \int _1^x\left(\log \frac{x}{y}\right)^{\alpha-1} \frac{f(y) dy}{y}, \ 1<x.$$
For further reading about fractional calculus we recommend, for example, \cite{KIR}.

%\vspace{1ex}

(vi) In 1988 S. Hilger introduced the calculus on time scales, a strong tool for unified treatment of differential and difference equations. Among different kinds of integrals the most investigated is $\Delta-$ integral, see  for example, \cite{ABoPet}.  A $\Delta-$ integral was followed by $\nabla-$ integral, $\diamondsuit_\alpha$ integral, $\alpha,\beta$-symmetric integral  etc. All of them  are isotonic linear functionals.  

\end{example}

\subsection*{Chebyshev-type inequalities for isotonic linear functionals}

After that short text about various kinds of isotonic linear functionals, let us say  few words about some inequalities of Chebyshev type involving isotonic linear functionals.

We say that functions $f$ and $g$ on $E$ are similarly ordered (or synchronous) if for each $x,y\in E$
$$(f(x)-f(y))(g(x)-g(y))\geq 0.$$
If the reversed inequality holds, then we say that  $f$ and $g$ are oppositely ordered or asynchronous. The most famous inequality which involve similarly or oppositely ordered functions is the Chebyshev inequality for integrals. It states that if $p, f$ and $g$ are   integrable   real functions on $[a,b]\subset {\bf R}$   and if $f$ and $g$ are similarly ordered, then
\begin{eqnarray}\label{cheb}
\int_a^b p(x) dx \int_a^b p(x) f(x) g(x) dx \geq \int_a^b p(x) f(x) dx \int_a^b p(x) g(x) dx.
\end{eqnarray}
If $f$ and $g$ are oppositely ordered then the reverse of the inequality in (\ref{cheb}) is valid.

During last century a lot of results about the Chebyshev  inequality appear. Here, we only give the most recent results involving two isotonic linear functionals, \cite{NV2}.

\begin{theorem}[The  Chebyshev inequality for two   functionals]
\label{thm:ceb}
Let $A$ and $B$ be two  isotonic linear functionals on $L$ and let $p, q\in L$ be  non-negative functions. Let  $f,g$ be two functions on $E$  such that $pf$, $pg$, $qf$, $qg, pfg$, $qfg \in L$.

If $f$ and $g$ are similarly ordered functions, then
\begin{eqnarray}
\label{ceb1}
A(pfg)B(q)+A(p)B(qfg)
 \geq A(pf)B(qg)+A(pg)B(qf) .
\end{eqnarray}
If $f$ and $g$ are oppositely ordered functions, then the reverse inequality in (\ref{ceb1}) holds.
\end{theorem}

Putting $A=B$, $p=q$ in (\ref{ceb1}) and divided by $2$ we get that for similarly ordered functions $f$ and $g$  
 such that $pf,pg,pfg \in L$,  the following holds
\begin{eqnarray*}
\label{ceb2}
A(p)A(pfg) \geq A(pf)A(pg).
\end{eqnarray*}
If $f$ and $g$ are oppositely ordered functions, then the reverse inequality  holds.

It is, in fact, the  Chebyshev inequality for one isotonic positive functional.

One of the most investigated question related to the Chebyshev integral inequality is the question of finding bounds for the so-called Chebyshev difference which is defined as a difference between two sides from inequality (\ref{cheb}). 
Results related to that question are called Gru\" uss type inequalities. In \cite{NV2} Gr\" uss type inequalities for the Chebyshev difference $T(A,B, p,q,f,g)$ which arise from inequality (\ref{ceb1}) are given 
where
\[
T(A,B, p,q,f,g) = 
   B(q)A(pfg)+A(p)B(qfg)-A(pf)B(qg)-A(pg)B(qf).
\]

\section{Inequalities for $M-g-$Lipschitz   and H\" older-type functions}

In this section $M-g-$Lipschitz functions are considered. We say that $f$ is an $M-g-$Lipschitz function if 
\[
|f(x)-f(y)| \leq M|g(x)-g(y)|\]
for all $x,y \in E$.
If   $g=id$, then $f$ is simple called   an $M-$Lipschitz function.

In the  following theorem we consider two functions $f$ and $g$ which are $h_1-$ and $h_2-$Lipschitz functions with  constants $M_1$ and $M_2$ respectively. 

\begin{theorem}\label{thm-lip}
\label{BT-thm3.4} Let $A$ and $B$ be isotonic linear functionals on $L$ and let
$p$, $q$ be non-negative functions from $L$.
Let $M_1$, $M_2$ be real numbers and let $f$, $g$, $h_1$, $h_2$ be functions such that
$f$ is $M_1-h_1-$Lipschitz  and $g$ is $M_2-h_2-$Lipschitz, i.e. for all $x,y\in E$ 
\begin{eqnarray}
\label{4.1}
| f(x)-f(y) | \leq M_1 |h_1(x)-h_1(y)|, && \\ \label{4.2}
| g(x)-g(y) | \leq M_2 |h_2(x)-h_2(y)|. && 
\end{eqnarray}
If all the terms in the below inequality exist and $h_1$ and $h_2$ are or similarly ordered, or oppositely ordered, then
\begin{eqnarray}\label{4.3}
| T(A,B,p,q, f,g) | \leq M_1M_2 T(A,B,p,q, h_1,h_2).
\end{eqnarray}
\end{theorem}

\begin{proof} Let $h_1$ and $h_2$ be similarly ordered.
Multiplying the inequalities (\ref{4.1}) and (\ref{4.2})
we get 
\[|(f(x)-f(y))(g(x)-g(y))| \leq M_1M_2(h_1(x)-h_1(y))(h_2(x)-h_2(y)).\]
It means that
\begin{eqnarray*}
\hspace*{-2pt} (f(x)-f(y))(g(x)-g(y)) \leq M_1M_2(h_1(x)-h_1(y))(h_2(x)-h_2(y)) &\textrm{and} &\\
 (f(x)-f(y))(g(x)-g(y)) \geq -M_1M_2(h_1(x)-h_1(y))(h_2(x)-h_2(y)) .&&
\end{eqnarray*}
Since 
\[(f(x)-f(y))(g(x)-g(y))=f(x)g(x)+f(y)g(y)-f(x)g(y)-f(y)g(x)
\]
multiplying with $p(x)q(y)$ and acting on the first inequality by functional $A$ with respect to $x$
and then by functional $B$ with respect to $y$  we get
 \begin{eqnarray*}
&&\hspace*{-46pt}A(pfg)B(q)+A(p)B(qfh)-A(pf)B(qg)-A(pg)B(qf) \\
\hspace*{-16pt} &\leq&  M_1M_2 (A(ph_1h_2)B(q)+A(p)B(qh_1h_2)-A(ph_1)B(qh_2)-A(ph_2)B(qh_1), i.e.
\end{eqnarray*}
\[T(A,B,p,q, f,g)  \leq M_1M_2 T(A,B,p,q, h_1,h_2).
\]
Similarly, from the second inequality we obtain
\[T(A,B,p,q, f,g)  \geq -M_1M_2 T(A,B,p,q, h_1,h_2)\]
and we get
the claimed result. The case when $h_1$ and $h_2$ are  oppositely ordered is proven similary.
%\qed 
\end{proof}

\begin{theorem}
\label{BT-thm3.3} Let $A$ and $B$ be isotonic linear functionals on $L$ and let
$p$, $q$ be non-negative functions from $L$.
Let $M$ be real number and let $f$, $g$ be functions such that
\[ | f(x)-f(y) | \leq M |g(x)-g(y)|, \qquad \forall x,y.
\]
If all the terms in the below inequality exist, then
\begin{eqnarray*}%\label{4.4}
\big| A(pfg)B(q)&+&A(p)B(qfh)-A(pf)B(qg)-A(pg)B(qf)\big|\\ \nonumber
& \leq& M \left(A(pg^2)B(q)-2A(pg)B(qg)+A(p)B(qg^2)\right).
\end{eqnarray*}
\end{theorem}

\begin{proof} 
Since $f$ is an $M-g-$Lipschitz function and $g$ is $1-g-$Lipschitz the desired inequality is a simple consequence of Theorem \ref{thm-lip}.
%\qed 
\end{proof}

 In  the following  table  
 we give a list of papers where particular cases of Theorems from this section can be found.
 
\begin{table}[h!t]
\caption{Known applications in particular cases of fractional integral operators}
\label{tab:2}
\noindent
\begin{tabular}{c||c|c}
Particular cases   & Theorem \ref{thm-lip}       & Theorem \ref{BT-thm3.3}                      \\ \hline \hline
  Riem.-Liouv.         &\cite{dttjns2010},            &\cite[Thm 3.5]{DAH4}, $A=B$, $p\not= q$,       \\  
        oper.           & $A=B$, $p\not= q$, $h_1=h_2=id$   &  \cite[Thm 3.7]{DAH4}, $A\not= B$, $p\not= q$, \\ \hline
Saigo oper.        &\cite[Thm 2.20]{Yang}           & \cite[Thm 2.19]{Yang}   \\  
                   &$A\not= B$, $p\not= q$, $h_1=h_2=id$ &$A\not= B$, $p\not= q$  \\ \hline
$q$-Saigo oper.    &\cite[Thm 3.17]{Yang}         & \cite[Thm 3.16]{Yang}   \\  
	           &$A\not= B$, $p\not= q$, $h_1=h_2=id$ &$A\not= B$, $p\not= q$,  \\ \hline
$q$-R-L.           &\cite[Thm 3.4]{BT} & \cite[Thm 3.3]{BT}   \\ 
   int.    operator  &$A\not= B$, $p\not= q$, $h_1=h_2=id$,  &$A\not= B$, $p\not= q$,   \\               
  &\cite[Thm 3.5]{BT} &                               \\ 
&$A= B$, $p\not= q$, $h_1=h_2=id$  &    \\ 
\hline
Riemann  & \cite[Thm 2.1]{D_Indian}, $h_1=h_2=id$& \cite[Thm 4.1]{D_Indian},   \\ 
   int.  & $A=B$, $p=q={\bf 1}$  & $A=B$, $p=q={\bf 1}$   \\ \hline
time scale & \cite[Thm 4.1]{BM2011},   & \cite[Thm 5.1]{BM2011},      \\  
$\diamond_\alpha$ integral & $A= B$, $p= q$, $h_1=h_2=id$  & $A= B$, $p= q$   \\ \hline
\end{tabular}
\end{table}

Let us say few words how to read the above table. 
%In the first column we write a list of isotonic linear functionals. In the corresponding row we give a reference where   applications of our Theorem \ref{basic}, Corollary \ref{SNT-cor5}  and Theorem \ref{SNT-thm18} occur. For example, Theorem \ref{basic} for two  Riemann-Liouville   operators, but with no weights can be find in paper \cite{TNS} as Theorem 2 etc.
In the first column we write a list of isotonic linear functionals. In the corresponding row we give a reference where   applications of our Theorem \ref{thm-lip}   and Theorem \ref{BT-thm3.3} occur. For example, Theorem \ref{thm-lip} for two  Saigo  operators, but with functions $h_1$ and $h_2$ equal to identity $id$ can be find in paper \cite{Yang} as Theorem 2.20 etc.

As we can see, our results enable us to give analogue results for other cases of linear functionals such as for fractional hypergeometric operators, for the Hadamard operators, for other kinds of integrals on time scales etc. Also, we improve existing results by using two different, more general functions $h_1$ and $h_2$ instead of identity function $id$. For example, here we give a result for the Hadamard integral operators. Let $p=q=1$, $h_1=h_2=id$, $\alpha, \beta >0$. If functions $f,g$ are $M_1-$, $M_2-$Lipschitz respectively, then the following inequality holds
$$
\frac{\log^\beta t}{\Gamma(\beta +1)}\phantom{a}_HJ^\alpha (fg(t)) + \frac{\log^\alpha t}{\Gamma(\alpha +1)}\phantom{a}_HJ^\beta (fg(t)) 
$$
$$-\phantom{a}_HJ^\alpha (f(t))\phantom{a}_HJ^\beta (g(t)) -\phantom{a}_HJ^\alpha (g(t))\phantom{a}_HJ^\beta (f(t))
\leq M_1M_2 \frac{t^2}{\Gamma(\alpha)\Gamma(\beta)}\times
$$
$$\times \left[\frac{\log^\alpha t \cdot \gamma(\beta, 2 \log t)}{2^\beta \alpha}+
\frac{\log^\beta t \cdot \gamma(\alpha, 2 \log t)}{2^\alpha \beta} -2\gamma(\alpha, \log t)\gamma(\beta, \log t) \right]
$$
where $\gamma(s,x)$ is incomplete Gamma function, i.e. $\gamma(s,x)=\int_0^x t^{s-1}e^{-t} dt$.

The procedure when we firstly  apply on some function $F(x,y)$     functional $A$ with respect to variable $x$
and then apply functional  $B$ with respect to variable $y$   occurs very often in this paper. So, we use the following notation: if $F(x,y)$ is a function, then the number which appears after the above-described procedure is written as
\[B_yA_x(F(x,y)) \ \ \mbox{or} \ \ B_yA_x(F).\]
It is worthless to say that if $A$ and $B$ are isotonic linear functionals, then a functional 
\[ F\mapsto B_yA_x(F)\]
is also an isotonic linear functional. 

\begin{theorem}
\label{holder}
Let $A$ and $B$ be isotonic linear functionals on
$L$ and let $p$, $q$ be non-negative functions from $L$. If $f$ is of $r$-H\" older-type and $g$ is of $s$-H\" older-type, i.e.
\[
|f(x)-f(y)| \leq H_1 |x-y|^r, \ \ \ |g(x)-g(y)| \leq H_2 |x-y|^s\]
for all $x,y\in E$, where $H_1,H_2 >0$, $r,s \in (0,1]$ fixed, then 
\[
|T(A,B,p,q, f, g)| \leq H_1H_2 \cdot B_yA_x (p(x)q(y)|x-y|^{r+s}).
\]
\end{theorem}

\begin{proof}  
It is proved in similar manner as Theorem \ref{thm-lip}.
%\qed 
\end{proof}

In particular case, when $p=q={\bf 1}$, $A(f)=B(f)=\int_a^b f(x)dx$, a factor $B_yA_x (p(x)q(y)|x-y|^{r+s})$ was calculated in \cite{D_Indian} and it is equal to 
\[
\frac{2(b-a)^{r+s+2}}{(r+s+1)(r+s+2)}.\]

\section{Inequalities for functions with variable bounds}

In this section we collect different results for functions with variable and constant bounds.
For example, if we look at the paper \cite{SNT}, their Theorem 4 can be seen as a particular case of the following theorem.

\begin{theorem} \label{basic}
Let $A$ and $B$ be isotonic linear functionals on $L$, $p$, $q$ be  non-negative functions from $L$.
Let $f$, $\phi_1, \phi_2, $ be functions such that 
\[\phi_1(t)\leq f(t) \leq \phi_2(t)\] 
and 
all terms in the below inequality exist. Then
\begin{eqnarray*}\label{3.1}
A(p\phi_2)B(qf)+A(pf)B(q\phi_1)\geq  A(p\phi_2)B(q\phi_1)+A(pf)B(qf).
\end{eqnarray*}
\end{theorem}

\begin{proof}  We consider the inequality
\[(\phi_2(x)-f(x))(f(y)-\phi_1(y))\geq 0.\] It is equivalent to the following
\[\phi_2(x)f(y)+f(x)\phi_1(y)\geq \phi_1(y) \phi_2(x)+f(x)f(y).\]
After multiplying with $p(x)q(y)$ and acting on this inequality first by functional $A$ with respect to $x$
and then by $B$ with respect to $y$  we get the wanted inequality.
\end{proof}
% \qed 

If functions $\phi_1, \phi_2 $ become   constant functions, then we  get the following corollary. 

\begin{corollary} \label{SNT-cor5}
Let $A$ and $B$ be isotonic linear functionals on $L$ and let
$p$, $q$ be  non-negative functions from $L$. Let $m$, $M$ be real numbers and let $f$ be function such that
 $$m\leq f(t)\leq M$$
 and all terms in the below inequality exist. Then
\[
MA(p)B(qf)+mA(pf)B(q)\geq   Mm A(p)B(q)+A(pf)B(qf).
\]
\end{corollary}

\begin{proof}  Follows from Theorem \ref{basic} for $\phi_1(t) = m$, $\phi_2(t)=M$. % \qed
\end{proof}

\begin{corollary} \label{SNT-cor6}
Let $A$ and $B$ be isotonic linear functionals on $L$ and let
$p$, $q$ be  non-negative functions from $L$. Let $M>0$ and let $\varphi$, $f$ be function such that
\[|f(t)-\varphi(t)|<M\]
 and all terms in the below inequality exist. Then
\begin{eqnarray*} \nonumber
&& \hspace*{-13pt} A(p\varphi) B(qf)+A(pf)B(q\varphi)+MA(p)B(qf)+MA(p\varphi)B(q)+M^2A(p)B(q)\\ 
%\label{3.3}
&&\geq A(p\varphi)B(q\varphi)+MA(p)B(q\varphi)+MA(pf)B(q)+A(pf)B(qf).
\end{eqnarray*}
\end{corollary}

\begin{proof}  Follows from Theorem \ref{basic} for $\phi_1(t) = \varphi(t)-M$, $\phi_2(t)=\varphi(t)+M$. %\qed
\end{proof}

 \begin{theorem}
\label{SNT-thm18}
Let $A$ and $B$ be isotonic linear functionals on $L$ and let $p$, $q$ be non-negative functions from $L$.
Let $\varphi_1, \varphi_2, $ $\psi_1, \psi_2,$ $f$ and $g$ be functions such that all terms in the below inequality exist and
conditions
\[\varphi_1(t)\leq f(t) \leq \varphi_2(t) \quad {\rm and}\quad \psi_1(t)\leq g(t) \leq \psi_2(t)\]
hold. Then
\begin{eqnarray*}
&&A(p\varphi_1)B(q\psi_1)+A(pf) B(qg) \geq A(p\varphi_1)B(qg) + A(pf)B(q\psi_1),\\
 &&A(p\varphi_1)B(q\psi_2)+A(pf) B(qg)\leq A(p\varphi_1)B(qg) + A(pf)B(q\psi_2),\\
 &&A(p\varphi_2)B(q\psi_1)+A(pf) B(qg)\leq A(p\varphi_2)B(qg) + A(pf)B(q\psi_1),\\
 &&A(p\varphi_2)B(q\psi_2)+A(pf) B(qg)\geq A(p\varphi_2)B(qg) + A(pf)B(q\psi_2).
\end{eqnarray*}
\end{theorem}

\begin{proof} 
The inequality $\left(f(x)-\varphi_1(x)\right)\left(g(y)-\psi_1(y)\right)\geq 0$
which can be written as
\[\varphi_1(x)\psi_1(y)+f(x)g(y)\geq \varphi_1(x)g(y)+f(x)\psi_1(y)\] is obviously true.

After multiplying it with $p(x)q(y)$ and acting on this inequality first by functional
$A$ with respect to $x$ and then by $B$ with respect to $y$,    we get 
\[
B_yA_x (p(x)q(y) (\varphi_1(x)\psi_1(y)  +f(x)g(y)) 
  \geq  
B_yA_x(p(x)q(y) (\varphi_1(x)g(y)+f(x)\psi_1(y))
\]
and applying properties of isotonic linear functionals $A$ and $B$ we obtain
the first inequality.
The other three inequalities are obtained in a similar way starting from
\begin{eqnarray*}
&&\left(f(x)-\varphi_1(x)\right)\left(\psi_2(y)-g(y)\right)\geq 0,\\
&&\left(\varphi_2(x)-f(x)\right)\left(g(y)-\psi_1(y)\right)\geq 0,\\
&&\left(\varphi_2(x)-f(x)\right)\left(\psi_2(y)-g(y)\right)\geq 0
\end{eqnarray*}
respectively.
%\qed 
\end{proof} 

\begin{corollary}
\label{SNT-cor19}
Let $A$ and $B$ be isotonic linear functionals on $L$ and let $p$, $q$ be non-negative functions from $L$.
Let $m,M,n,N$ be real numbers and let $f$ and $g$ be functions such that
\[m\leq f(t) \leq M \quad {\rm and}\quad n\leq g(t) \leq N\]
and all terms in the below inequalities exist. Then
\begin{eqnarray*}
&& mn A(p)B(q)+A(pf) B(qg)\geq m A(p)B(qg) + n A(pf)B(q),\\
&& mN A(p)B(q)+A(pf) B(qg)\leq m A(p)B(qg) + N A(pf)B(q),\\
&& Mn A(p)B(q)+A(pf) B(qg)\leq M A(p)B(qg) + n A(pf)B(q),\\
&& MN A(p)B(q)+A(pf) B(qg)\geq M A(p)B(qg) + N A(pf)B(q).
\end{eqnarray*}
\end{corollary}

\begin{proof}  Follows from Theorem \ref{SNT-thm18} for $\phi_1=m$, $\phi_2=M$, $\psi_1=n$, $\psi_2=N$. %\qed 
\end{proof}

\begin{theorem}
\label{SNT-thm8}
Let $A$ and $B$ be isotonic linear functionals on $L$ and let $p$, $q$ be  non-negative functions from $L$. Let $\theta_1$ and $\theta_2$ be a positive real numbers satisfying $\displaystyle \frac{1}{\theta_1}+ \frac{1}{\theta_2} =1$.
Let $f$, $\phi_1, \phi_2, $ be functions such that 
 \[\phi_1(t)\leq f(t) \leq \phi_2(t).\]
and all terms in the below inequality exist. Then
\begin{eqnarray*} \nonumber
 &&\frac{1}{\theta_1}   B(q) A(p(\phi_2-f)^{\theta_1})+ \frac{1}{\theta_2}A(p)B(q(f-\phi_1)^{\theta_2})+ A(p\phi_2)B(q\phi_1)\\ 
&&+A(pf)B(qf)  \geq  A(p\phi_2)B(qf)+A(pf)B(q\phi_1).
\label{young1}
\end{eqnarray*}
\end{theorem}

\begin{proof} 
Let us mention the Young inequality which holds for non-negative $a,b$ and for  positive  $\theta_1$ and $\theta_2$ with property $\displaystyle \frac{1}{\theta_1}+ \frac{1}{\theta_2} =1$:
 \[  \frac{1}{\theta_1} a^{\theta_1}+ \frac{1}{\theta_2} b^{\theta_2} \geq ab.
\]
Setting in the previous inequality
\[a=\phi_2(x) -f(x), \ \ \ b=f(y)-\phi_1(y)\]
we have
\[
  \frac{1}{\theta_1} (\phi_2(x) -f(x))^{\theta_1}+ \frac{1}{\theta_2} (f(y)-\phi_1(y))^{\theta_2} \geq (\phi_2(x) -f(x))(f(y)-\phi_1(y)).\]
Applying usual procedure we get
\begin{eqnarray*} 
&& B_yA_x \left(p(x)q(y) (\frac{1}{\theta_1} \right.   \left. (\phi_2(x) -f(x))^{\theta_1}+ \frac{1}{\theta_2} (f(y)-\phi_1(y))^{\theta_2}\right)\\
&&\hspace{15mm} +B_yA_x \Big(p(x)q(y) (\phi_2(x) -f(x))(f(y)-\phi_1(y)\Big)
\end{eqnarray*}
and after applying properties of $A$ and $B$ we obtain
\[\frac{1}{\theta_1} B(q) A(p(\phi_2-f)^{\theta_1})+ \frac{1}{\theta_2}A(p)B(q(f-\phi_1)^{\theta_2}) 
 \geq A(p(\phi_2 -f))B(q(f-\phi_1)).\]
Using result from Theorem \ref{basic} we get inequality (\ref{young1}).
%\qed 
\end{proof}

 \begin{corollary}
\label{SNT-cor9}
Let $A$ and $B$ be isotonic linear functionals on $L$ and let $p$, $q$ be non-negative functions from $L$.
Let $m, n\in{\bf R}$ and let $f$ be function such that   $m\leq f(t)\leq M$ and all terms in the below inequality exist.
Then
\begin{eqnarray*}
(M+m)^2 A(p)B(q) &+ &A(pf^2)B(q)+2A(pf)B(qf)+A(p)B(qf^2)\\
&\geq&  2(M+m)[A(p)B(qf)+A(pf)B(q)].
\end{eqnarray*}
\end{corollary}

\begin{proof}  Follows from Theorem \ref{SNT-thm8} for $\phi_1=m$, $\phi_2=M$ and $\theta_1=\theta_2=2$. 
%\qed 
\end{proof}

\begin{theorem}
\label{SNT-thm20}
Let $A$ and $B$ be isotonic linear functionals on $L$ and let $p$, $q$ be  non-negative functions from $L$.
Let $\theta_1$ and $\theta_2$ be a positive real numbers satisfying $\displaystyle \frac{1}{\theta_1}+ \frac{1}{\theta_2} =1$.
Let $\varphi_1, \varphi_2, $ $\psi_1, \psi_2,$ $f$ and $g$ be functions such that all terms in the below inequality exist and
conditions
\[\varphi_1(t)\leq f(t) \leq \varphi_2(t) \quad {\rm and}\quad \psi_1(t)\leq g(t) \leq \psi_2(t)\]
hold. Then
\begin{eqnarray*}
&&\hspace*{-38pt} \frac{1}{\theta_1} A\left(p (\varphi_2-f)^{\theta_1}\right) B(q) + \frac{1}{\theta_2} A(p) B\left(q(\psi_2-g)^{\theta_2}\right) 
\geq A \left(p(\varphi_2-f)\right) B\left(q(\psi_2-g)\right),\\ 
&&\hspace*{-30pt} \frac{1}{\theta_1} A\left(p (\varphi_2-f)^{\theta_1}\right) B(q)+ \frac{1}{\theta_2} A(p) B\left(q(g-\psi_1)^{\theta_2}\right)
\geq A \left(p(\varphi_2-f)\right) B\left(q(g-\psi_1)\right),\\ 
&&\hspace*{-30pt} \frac{1}{\theta_1} A\left(p (f-\varphi_1)^{\theta_1}\right) B(q)+ \frac{1}{\theta_2} A(p) B\left(q(\psi_2-g)^{\theta_2}\right)
\geq A \left(p(f-\varphi_1)\right) B\left(q(\psi_2-g)\right),\\ 
&&\hspace*{-30pt} \frac{1}{\theta_1} A\left(p (f-\varphi_1)^{\theta_1}\right) B(q)+ \frac{1}{\theta_2} A(p) B\left(q(g-\psi_1)^{\theta_2}\right)
\geq A \left(p(f-\varphi_1)\right) B\left(q(g-\psi_1)\right).
\end{eqnarray*}
\end{theorem}

\begin{proof} 
Using the Young inequality for $a=\varphi_2(x) - f(x)$, $b=\psi_2(y)-g(y)$ we get
\[\frac{1}{\theta_1} \left( \varphi_2(x)-f(x)\right)^{\theta_1}+ \frac{1}{\theta_2} \left(\psi_2(y)-g(y)\right)^{\theta_2})
 \geq \left(\varphi_2(x)-f(x)\right) \left(\psi_2(y)-g(y)\right).
\]
Multiplying both sides with $p(x)q(y)$ and acting on the inequality by functional $A$ with respect to $x$ and then by $B$ with respect to $y$,
we get
\[
\frac{1}{\theta_1} A\left( p (\varphi_2-f)^{\theta_1}\right) B(q)+ \frac{1}{\theta_2} A(p) B\left(q(\psi_2-g)^{\theta_2}\right)\]
\[
 \geq A \left(p(\varphi_2-f)\right) B\left(q(\psi_2-g)\right).
\]
In the similar way the other three inequalities are proved.
%\qed 
\end{proof}

 In the following table we give a list of papers where particular cases of some Theorems from this section can be found.

%\noindent
\begin{table}[h!t]
\caption{Known applications for particular cases of fractional integral operators}
\label{tab:1}
\begin{tabular}{c||c|c|c}
Particular cases   & Theorem \ref{basic} & Corollary \ref{SNT-cor5}  & Theorem \ref{SNT-thm18}  \\ \hline \hline
  R-L oper. &\cite[Thm 2]{TNS},  &\cite[Cor 3]{TNS},  & \cite[Thm 5]{TNS},   \\  
   & $A\not=B$, $p= q={\bf 1}$ & $A\not=B$, $p= q={\bf 1}$ &$A\not=B$, $p= q={\bf 1}$ \\
                \hline
E-K oper. & \cite[Cor 2]{WHPG}& &\cite[Cor 3]{WHPG} \\ 
          &$A\not=B$, $p= q={\bf 1}$ & &$A\not=B$, $p= q={\bf 1}$ \\ \hline
Saigo oper. & \cite[Thm 8]{CP} & \cite[remark 9]{CP} & \cite[Thm 10]{CP}\\ 
            & $A\not=B$, $p= q={\bf 1}$ &$A\not=B$, $p= q={\bf 1}$  &$A\not=B$, $p= q={\bf 1}$ \\ 
&\cite[Thm1]{WHPG} & \cite[Cor 4]{WHPG} & \cite[Thm2]{WHPG}\\ 
& $A\not=B$, $p= q={\bf 1}$ &$A\not=B$, $p= q={\bf 1}$  & $A\not=B$, $p= q={\bf 1}$\\ \hline
Had. oper. &\cite[Thm 4]{SNT},  &\cite[Cor 5]{SNT},  & \cite[Thm 18]{SNT},  \\ 
& $A\not=B$, $p= q={\bf 1}$ &$A\not=B$, $p= q={\bf 1}$  &$A\not=B$, $p= q={\bf 1}$ \\ \hline
\end{tabular}
\end{table}
%\smallskip

   Let us mention that the  non-weighted versions of Theorems \ref{SNT-thm8}, \ref{SNT-thm20} and Corollary \ref{SNT-cor9} for two Hadamard operators $A=_HJ^\alpha$ and $B=_HJ^\beta$ can be found in paper \cite{SNT}. 

As we can see, we did not find similar results for hypergeometric operators in literature. But, it is obvious that our results can be applied on two fractional hypergeometric operator or on $q$-analogues of those integral operators.

\section{Inequalities for three functions }

This section is devoted to the results involving three or more functions and in some sense it generalize the previous section.

\begin{theorem}
\label{RA-thm2}
Let $A$ and $B$ be isotonic linear functionals on $L$ and let $p$, $q$ be non-negative functions from $L$.
Let $f$, $g$ be similarly ordered functions and let $h$ be function with positive values. If all terms in the below inequality exist, then
\begin{eqnarray}\nonumber
&& A(pfgh)B(q)+ A(pfg)B(qh)+A(ph)B(qfg)+ A(p)B(qfgh) \\ \label{ffgghh}
&&\geq A(pfh)B(qg) + A(pf)B(qgh)+ A(pgh)B(qf) + A(pg)B(qfh).
\end{eqnarray}
If $f$ and $g$ are oppositely ordered, then  the reversed inequality holds.
\end{theorem}

\begin{proof} 
For similarly ordered functions $f$, $g$ we have $(f(x)-f(y))(g(x)-g(y))\geq 0$, but then also
\[(f(x)-f(y))(g(x)-g(y))(h(x)+h(y))\geq 0.\]
This can be written as
\begin{eqnarray*}
&&f(x)g(x)h(x)+f(x)g(x)h(y)+f(y)g(y)h(x)+f(y)g(y)h(y)\\
&\geq& f(x)g(y)h(x) + f(x)g(y)h(y)+ f(y)g(x)h(x)+f(y)g(x)h(y).
\end{eqnarray*}
Multiplying both sides by $p(x)q(y)$ and acting on this inequality first by functional
$A$ with respect to $x$ and then by $B$ with respect to $y$  we get
the desired inequality.
%\qed 
\end{proof}

\begin{remark}
For $p=q$ from previous theorem we get
\begin{eqnarray*}
&& A(pfgh)B(p)+ A(pfg)B(ph)+A(ph)B(pfg)+ A(p)B(pfgh)\\
&\geq& A(pfh)B(pg) + A(pf)B(pgh)+ A(pgh)B(pf) + A(pg)B(pfh).
\end{eqnarray*}
 
If also $A=B$ then
\[
A(pfgh)A(p)+ A(pfg)A(ph) \geq A(pfh)A(pg) + A(pf)A(pgh)
\]
and for $h=const$ it reduces to the Chebyshev inequality.
\end{remark}

\begin{remark}
Particular cases of inequality (\ref{ffgghh})   are appeared  in several papers for different kinds of linear functionals. For example, if $A$ and $B$ are Riemann-Liouville's operators, then a  non-weighted inequality  is given in Theorem 2.1 in paper \cite{sulaiman}. If $A$ and $B$ are different fractional $q$-integral of the Riemann-Liouville-type, then (\ref{ffgghh}) is given in \cite[Thm 2.1]{sroy} for $p=q$. %An analogue non-weighted result for $k$-Riemann-Liouville fractional integrals is given in \cite[Thm 6]{sarikaya}.   
Inequalities involving two $q$-analogues of Saigo's fractional integral operators which are particular cases of (\ref{ffgghh}) are given in \cite{BA} as Theorems 5 and 6, while similar results for generalized $q$-Erd\' elyi-Kober fractional integral operators are given in \cite[Thm 1 and 2]{RA}.
\end{remark}
 
\begin{lemma}\label{lemma5.1}
Let $A$ and $B$ be isotonic linear functionals on
$L$ and let $p$, $q$ be non-negative functions from $L$. Let
\begin{eqnarray*}
\hspace*{-8pt}H_{f,g,h}(x,y)&= &\left(f(x)-f(y)\right)\left(g(x)-g(y)\right)\left(h(x)-h(y)\right)\\
&= &f(x)g(x)h(x)+f(x)g(y)h(y)+f(y)g(x)h(y)+f(y)g(y)h(x)\\
& -& f(y)g(x)h(x)-f(x)g(y)h(x)- f(x)g(x)h(y)-f(y)g(y)h(y).
\end{eqnarray*}
Then
\begin{eqnarray*}
&& \hspace*{-8pt} A_x B_y  (p(x)q(y) H_{f,g,h}(x,y)) \\
&= &A(pfgh)B(q)+A(pf)B(qgh)+A(pg)B(qfh)+A(ph)B(qfg)\\
& -& A(pgh)B(qf)-A(pfh)B(qg)- A(pfg)B(qh)-A(p)B(qfgh).
\end{eqnarray*}
\end{lemma}

\begin{theorem}
\label{RA-thm4}
Let $A$ and $B$ be isotonic linear functionals on $L$ and let $p$, $q$ be non-negative functions from $L$.
Let $m,M$, $n, N$, $k, K$ be real numbers and let $f$, $g$, $h$ be functions such that 
$$m\leq f(x)\leq M, \quad n\leq g(x)\leq N, \quad k\leq h(x)\leq K,\quad \forall x.$$
If all the terms in the below inequality exist, then
\begin{eqnarray*}
&&\big| A(pfgh)B(q)+A(pf)B(qgh)+A(pg)B(qfh)+A(ph)B(qfg) \\
&&-A(pfg)B(qh)-A(pfh)B(qg)-A(pgh)B(qf)-A(p)B(fgh)\big|\\
&& \leq (M-m)(N-n)(K-k) A(p)B(q).
\end{eqnarray*}
\end{theorem}

\begin{proof} 
From $m\leq f(x)\leq M$ it follows $\left| f(x)-f(y)\right|\leq M-m$.
Therefore
\[
\left| \left(f(x)-f(y)\right)\left(g(x)-g(y)\right)\left(h(x)-h(y)\right) \right| \leq (M-m)(N-n)(K-k), \  i.e.
\]
\[
|H_{f,g,h}(x,y) | \leq (M-m)(N-n)(K-k). 
\]
 Multiplying both sides by $p(x)q(y)$ and using Lemma \ref{lemma5.1} we get 
\begin{eqnarray*}
&&\Big|
A(pfgh) B(q)-A(pfg)B(qh)-A(pfh)B(qg)+A(pf)B(qgh) \\
&& -A(pgh)B(qf)+A(pg)B(qfh)+A(ph)B(qfg)-A(p)B(fgh)
\Big| \\
&&\leq  (M-m)(N-n)(K-k) A(p)B(q),
\end{eqnarray*}
which proves the theorem.
%\qed 
\end{proof}

\begin{remark}
Particular cases of Theorem \ref{RA-thm4}  are appeared  in \cite[Thm 8 and 9]{BA} for two $q$-analogues of Saigo's fractional integral operators and in \cite[Thm 3 and 4]{RA} for generalized $q$-Erd\' elyi-Kober fractional integral operators.
\end{remark}
  
 \begin{theorem}
\label{RA-thm6}
Let $A$ and $B$ be isotonic linear functionals on $L$ and let $p$, $q$ be non-negative functions from $L$.
Let $M_1,M_2,M_3$ be real numbers and let $f_i$, $(i=1,2,3)$ be   $M_i-g-$Lipschitz functions.
  If all the terms in the below inequality exist, then
\begin{eqnarray*}
&&\big| A(pf_1f_2f_3)  B(q)+A(pf_1)B(qf_2f_3)+A(pf_2)B(qf_1f_3)+A(pf_3)B(qf_1f_2) \\
&&-A(pf_1f_2)B(qf_3)-A(pf_1f_3)B(qf_2)-A(pf_2f_3)B(qf_1)-A(p)B(qf_1f_2f_3)\big|\\
&& \leq M_1M_2M_3 \cdot B_yA_x(p(x)q(y)|g(x)-g(y)|^3).
\end{eqnarray*}
\end{theorem}

\begin{proof}  If $f_i$, $(i=1,2,3)$ are   $M_i-g-$Lipschitz functions, then 
\[
|f_1(x)-f_1(y)|\leq M_1|g(x)-g(y)|,\quad 
|f_2(x)-f_2(y)|\leq M_2|g(x)-g(y)|,\]
\[|f_3(x)-f_3(y)|\leq M_3|g(x)-g(y)| \]
for all $x,y$.
Multiplying those inequalities we get 
\[|H_{f_1,f_2,f_3}(x,y)|\leq M_1M_2M_3|g(x)-g(y)|^3.\]
This is equivalent to
\[H_{f_1,f_2,f_3}(x,y)\leq M_1M_2M_3|g(x)-g(y)|^3 \quad {\rm and} \]
\[ -H_{f_1,f_2,f_3}(x,y)\leq M_1M_2M_3|g(x)-g(y)|^3.\]
Multiplying both inequalities with $p(x)q(y)$ and acting on the resulting inequalities by $A$ with respect to $x$ and then by $B$ with respect to $y$,
we get the desired result.
%\qed 
\end{proof}

\begin{remark}
In \cite[Thm 11 and 12]{BA}  and \cite[Thm 5 and 6]{RA} authors attempted to give corresponding results for $q$-analogues of Saigo's fractional integral operators and  for generalized $q$-Erd\' elyi-Kober fractional integral operators, respectively. But they used assumptions $|f_i(x)-f_i(y)|\leq M_i(x-y)$, $i=1,2,3$, $x,y>0$ which leads to conclusion that $f_i \equiv 0$.
\end{remark}

\begin{remark}
Considering results from this and from the previous section it is clear how Theorems \ref{thm-lip} and \ref{RA-thm6} can be generalized for $n$ $M_i-g-$Lipschitz functions $f_i$, $i=2, \ldots ,n$. We leave it to a reader.
\end{remark}

\subsection*{Results with two functions and three weights} The following result is based on the succesive using of the Chebyshev inequality for pairs of weights. 

\begin{theorem} \label{fgpqr}
Let $A$ and $B$ be isotonic linear functionals on $L$ and let $p,q,r$ be non-negative functions from $L$. If $f$ and $g$ are similarly ordered functions, then 
\begin{eqnarray*}
A(p)[2A(q)B(rfg)&+&
A(r)B(qfg)+B(r)A(qfg)] \\
 +A(pfg)[A(q)B(r)+A(r)B(q)]&\geq&
A(p)[A(qf)B(rg)+A(qg)B(rf)]\\
+
A(q)[A(pf)B(rg)+A(pg)B(rf)]&+&
A(r)[A(pf)B(qg)+A(pg)B(qf)],
\end{eqnarray*}
under assumptions that all terms are well-defined.

If $f$ and $g$ are oppositely ordered functions, then the reversed inequality holds.
\end{theorem}

\begin{proof} 
Replacing in (\ref{ceb1}) $p$ by $q$ and $q$ by $r$ and multiplying by $A(p)$ we get
\[
A(p)[A(q)B(rfg)+B(r)A(qfg)] \geq A(p)[A(qf)B(rg)+A(qg)B(rf)].
\]
Replacing in (\ref{ceb1})  $q$ by $r$ and multiplying by $A(q)$ we get
\[
A(q)[A(p)B(rfg)+B(r)A(pfg)] \geq A(q)[A(pf)B(rg)+A(pg)B(rf)].
\]
Multiplying (\ref{ceb1}) by $A(r)$ we get
\[
A(r)[A(p)B(qfg)+B(q)A(pfg)] \geq A(r)[A(pf)B(qg)+A(pg)B(qf)].
\]
Adding the above inequalities we get the statement of the theorem.
%\qed 
 \end{proof}

\begin{remark}
Theorem \ref{fgpqr} are proved in several papers for different kinds of linear operators. For example, if $A$ and $B$ are the Riemann-Liouville operators, then it is given in \cite{DAH5}. Result involving Hadamard operators is given in \cite{CP1}, while an analogue result for the Saigo operators and $q$-analogue of Saigo's operators are given in \cite{choi} and \cite{Yang}.
\end{remark}

%promijenjen broj ugovora u odnosu na poslani file

\section*{Acknowledgements}
  The research of the second author was partially supported by the Sofia University SRF under contract No 146/2015.

%\bibliography{bnv}

\begin{thebibliography}{99}

\bibitem{BPA}   Baleanu, D,  Purohit, SD,  Agarwal, P.: On fractional integral inequalities involving hypergeometric operators. Chinese Journal of Mathematics.
{\bf 2014}, Article ID 609476 (2014).

\bibitem{DAH2}   Belarbi, S, Dahmani, Z: On some new fractional integral inequalities. Journal JIPAM. {\bf 19}(3) Art. 86 (2009). 

\bibitem{CP1}  Chinchane, VL,  Pachpatte, DB: On some integral inequalities using Hadamard fractional integral. Malaya Journal of Matematik. {\bf 1}(1) 62--66 (2012).

\bibitem{CP}  Chinchane, VL,  Pachpatte, DB: On some Gr\" uss-type fractional inequalities using Saigo fractional integral operator. Journal of Mathematics. {\bf 2014}, Article ID 527910 (2014).

\bibitem{SNT}  Sudsutad, W,   Ntouyas, SK,  Tariboon, J: Fractional integral inequalities via Hadamard's fractional integral. Abstract and Applied Analysis.
{\bf 2014}, Article ID 563096 (2014). 

\bibitem{TNS}  Tariboon, J,  Ntouyas, SK,  Sudsutad, W: Some new Riemann-Liouville fractional integral inequalities. International Journal of Mathematics and Mathematics Sciences. {\bf 2014}, Article ID 869434 (2014).

\bibitem{WHPG}  Wang, G,  Harsh, H,  Purohit, SD,  Gupta, T: A note on Saigo's fractional integral inequalities. Turkish Journal of Analysis and Number Theory.  {\bf 2}(3), 65--69 (2014).

\bibitem{proschan}  Pe\v{c}ari\'{c}, JE,  Proschan, F,  Tong, YL:   Convex functions, partial orderings, and statistical applications.
 Academic Press Inc. (1992). %S


\bibitem{NV}  Nikolova, L,  Varo\v sanec, S: Properties of mappings generated with inequalities for isotonic linear functionals. Proceedings of 'The International Conference Constructive Theory of Functions, Sozopol 2013'.  Sofia (Bulgaria): Prof. Marin Drinov Academic Publishing House;   p. 199--215 (2014).

\bibitem{NV2}   Nikolova, L,    Varo\v sanec, S: Chebyshev and Gr\" uss type inequalities involving two linear functionals and applications. Mathematical Inequalities and Applications, to appear. %S


\bibitem{PT1}  Pe\v cari\' c, J,  Tepe\v s, B: On a Gr\" uss type inequality for isotonic linear functionals I. Nonlinear Studies. {\bf 12}(2), 119--125 (2005).
%tu sam stala
\bibitem{BA}   Baleanu, D,    Agarwal, P: Certain inequalities involving the fractional $q$-integral operators. Abstract and Applied Analysis. {\bf 2014}, Article ID 371274 (2014).


\bibitem{choi}  Choi, J,   Agarwal, P: Some new Saigo type fractional integral inequalities and their $q$-analogues. Abstract and Applied Analysis.  {\bf 2014}, Article ID 579260 (2014).



\bibitem{AA} Anastassiou, GA: Fractional Differentiation Inequalities. Springer, Dordrecht-Heidelberg-London-New York  (2009).


\bibitem{KIR}  Kiryakova, V: Generalized Fractional Calculus and Applications. Pitman Research Notes in Math. Series, 301. New York (USA): Longman and J. Wiley; (1994).


\bibitem{SAMKO}  Samko SG,  Kilbas AA,   Marichev OI: Fractional Integrals and Derivatives: Theory and Applications.  Gordon and Breach, Yverdon(Switzerland)  (1993).

\bibitem{ABoPet}  Agarwal, R,  Bohner, M, Peterson, A: Inequalities on time scales: a survey. Mathematical Inequalities and Applications. {\bf 4}, 535--557 (2001).

\bibitem{BP}   Bohner, M,   Peterson, A: Dynamic Equations on Time Scales. Birkh\" auser (2001).

\bibitem{dttjns2010}   Dahmani, Z,  Tabharit, L,  Taf, S: Some fractional integral inequalities. J. Nonlinear Science. Lett A, {\bf 1}(2), 155--166 (2010).

\bibitem{DAH4}   Dahmani, Z: Some results associated with fractional integrals involving the extended Chebyshev functional. Acta Universitatis Apulensis. {\bf 27}, 217--224 (2011). 

\bibitem{Yang}  Yang, W: Some new Chebyshev and Gr\" uss-type integral inequalities for Saigo fractional integral operators and their $q$-analogues. Filomat. {\bf 29}(6), 1269--1289 (2015).

\bibitem{BT}   Brahim, K, Taf, S: On some fractional q-integral inequalities. Malaya Journal of Matematik. {\bf 3}(1), 21--26 (2013).

\bibitem{D_Indian}  Dragomir, SS: Some integral inequalities of Gr\" uss type. Indian J. Pure Appl. Math. {\bf 31}(4), 397--415 (2000).

\bibitem{BM2011}  Bohner, M,  Mathews, T,  Tuna, A: Diamond-alpha Gr\" uss type inequalities on time scales. Int. J. Dynamical Systems and Differential Equation. {\bf 3}(1/2), 234--247 (2011).

\bibitem{sulaiman}  Sulaiman, WT: Some new fractional integral inequalities. Journal of Mathematical Analysis. {\bf 2}(2), 23--28 (2011).

\bibitem{sroy}  Sroysang, B: A study on a new fractional integral inequality in quantum calculus. Adv. Studies Theor. Phys. {\bf 7}(14), 689--692 (2014).

\bibitem{sarikaya}  Sarikaya, MZ,  Karaca, A: On the k-Riemann-Liouville fractional integral and applications. International Journal of Statistics and Mathematics. 
{\bf 1}(3), 33--43 (2014).

\bibitem{RA}  Ritelli D,   Agarwal, P: On some new inequalities involving generalized Erd\' elyi-Kober fractional q-integral operator. arXiv:1405.6829v1.

\bibitem{DAH5}  Dahmani, Z: New inequalities in fractional integrals. International Journal of Nonlinear Science.
{\bf 9}(4), 493--497 (2010).
 
\end{thebibliography}
%\bibliographystyle{mmn}

%\end{document}

\end{document}